\documentclass[reqno,a4paper]{amsart}

\usepackage[mathscr]{eucal}
\usepackage{amssymb,verbatim,graphicx,epstopdf,soul,tikz}
\usepackage{enumerate}
\usepackage{graphpap}
\usepackage{mathrsfs}
\usetikzlibrary{hobby}

\newtheorem{thm}{Theorem}

\newtheorem{lem}[thm]{Lemma}

\newtheorem{cor}[thm]{Corollary}

\newtheorem{defn}[thm]{Definition}

\theoremstyle{definition}
\newtheorem{eg}[thm]{Example}

\newtheorem{remark}[thm]{Remark}
\newtheorem{obs}[thm]{Observation}

\newcommand{\lrge}{\operatorname{large}}
\newcommand{\CSP}{\operatorname{CSP}}
\newcommand{\QMEM}{\operatorname{QMEM}}

\newcommand{\up}{\textup}

\newcommand{\set}{\operatorname{set}}

\begin{document}
\title[Universal Horn classes of Hypergraphs]{Axiomatisability and hardness for universal Horn classes of hypergraphs}

\begin{abstract}
We characterise finite axiomatisability and intractability of deciding membership for universal Horn classes generated by finite loop-free hypergraphs.
\end{abstract}
\author[L. Ham]{Lucy Ham}
\address{Department of Mathematics and Statistics, La Trobe University VIC 3086, Australia}
\email{leham@students.latrobe.edu.au}
\author[M. Jackson]{Marcel Jackson}
\address{Department of Mathematics and Statistics, La Trobe University VIC 3086, Australia}
\email{m.g.jackson@latrobe.edu.au}
\thanks{The second author was supported by ARC Discovery Project DP1094578 and Future Fellowship FT120100666}
\maketitle


A \emph{universal Horn class} is a class of model-theoretic structures of the same signature, closed under taking ultraproducts ($\textsf{P}_{\rm u}$), direct products over nonempty families ($\textsf{P}$) and isomorphic copies of substructures ($\textsf{S}$); see \cite{bursan,gor,mal,mcn} for example.  Equivalently they are classes axiomatisable by way of \emph{universal Horn sentences}: universally quantified disjunctions $\alpha_1\vee\dots\vee \alpha_k$, where each $\alpha_i$ is either an atomic formula of the language, or a negated atomic formula, and all but at most one of the $\alpha_i$ are negated.   \emph{Quasivarieties} are very closely related classes, differing from the universal Horn class definition only in that the trivial one-element structure (in which all relations are total) is automatically included; this corresponds to allowing the degenerate direct product over an empty family of structures.  

Problems of axiomatisability for universal Horn classes and quasivarieties have a relatively long history.  The starting point is perhaps Maltsev's characterisation of semigroups embeddable in groups \cite{mal:sgp,mal:sgp2}, with subsequent developments in semigroup theory including Sapir~\cite{sap}, Margolis and Sapir~\cite{marsap}, Jackson and Volkov~\cite{jacvol}.  There is also a wealth of literature within universal algebra and relational structures; see Gorbunov's book \cite{gor}, or the Studia Logica special issue \cite{ABISV} for many examples.  An extra impetus for investigation of universal Horn classes comes from  computational complexity.  For example, the fixed template constraint satisfaction problem over a finite relational structure is the problem of deciding membership of relational structures in a certain universal Horn class~\cite{jactro}.  Computational issues for universal Horn classes of relational structures also play a hidden role behind a number of examples demonstrating intractability of deciding membership of finite algebras in a finitely generated pseudovariety.  Indeed, several of the relatively few known examples involve encoding a \texttt{NP}-complete universal Horn class membership problem into a pseudovariety membership problems.  This is true for Szekely \cite{sze}, Jackson and McKenzie~\cite{jacmck} and \cite{jac:SAT} for example.  

The present article concerns both axiomatisability and computational complexity for universal Horn classes of loop-free hypergraphs, and we are able to extend all of the known results for simple graphs.  
The characterisation of finitely axiomatisable universal Horn classes of finite simple graphs was given by Caicedo~\cite{cai}, by combining a probabilistic result of Erd\H{o}s \cite{erd} with work of Ne\v{s}et\v{r}il and Pultr~\cite{nespul}.  In fact, Caicedo's work covers any universal Horn class whose members have bounded chromatic number.  After fixing a reasonable model-theoretic meaning to ``hypergraph'' we show that Caicedo's classification may be extended to arbitrary loop-free hypergraphs.  The precise statement depends on technicalities concerning how hyperedges are to be recorded as relations, but even without this it is possible to state an abridged version of the result as follows.

\begin{thm}\label{thm:main}
Let $\mathscr{H}$
 be a universal Horn class of hypergraphs without singleton hyperedges and with bounded chromatic number and hyperedge cardinality.  If $\mathscr{H}$ consists of disjoint unions of bipartite graphs---including the degenerate cases where there are no hyperedges---then $\mathscr{H}$ has a finite axiomatisation in first order logic.  In all other cases,  $\mathscr{H}$ has no finite axiomatisation in first order logic.  
\end{thm}
In particular, $\mathscr{H}$ can never have a finite axiomatisation if it contains a hypergraph with at least one hyperedge of arity more than $2$.    
As in the case of Caicedo's classification, the argument for  nonfinite axiomatisability will again follow by probabilistic constructions (this time Erd\H{o}s and Hajnal \cite{erdhaj}), while we are able to show that the finitely axiomatisable case becomes almost completely degenerate.  These results are found in Section \ref{sec:hypergraph}.  In Section \ref{lem:continuum} we extend another result in \cite{cai} by showing that there are continuum many universal Horn classes of hypergraphs.  We apply a method of Bonato~\cite{bon} to show that every interval in the homomorphism order on hypergraphs represents a continuum of universal Horn classes; this requires a new extension to a result of Ne\v{s}et\v{r}il \cite{nes} on the density of the homomorphism order on hypergraphs.  
In Section \ref{sec:finite} we turn to the question of axiomatisability amongst finite structures, a topic that has generated quite a lot of interest in finite model theory; see~\cite{ros} for example.  
We are able to show that Theorem~\ref{thm:main} continues to hold when restricted to finite structures only: this appears to be new even in the case of simple graphs, and shows that the model-theoretic \textsf{SP}-Preservation Theorem holds for classes of hypergraphs of bounded chromatic number and hyperedge cardinality.
This is established using an Ehrenfeucht-Fra\"{\i}ss\'e game argument to observe a general lemma applying to any hereditary class of finite structures that is closed under certain disjoint unions.  
Finally, in Section~\ref{sec:hard} we observe an alternative path to the results of Sections~\ref{sec:hypergraph} and~\ref{sec:finite}  by application of the authors' recent 
All or Nothing Theorem~\cite{hamjac}; see Theorem~\ref{thm:hard} below.  
This approach has the advantage of adding complexity-theoretic hardness results for associated computational problems and avoiding the probabilistic constructions completely.  
However, the method depends on the All or Nothing Theorem, whose proof is substantially more involved than the direct constructions here.

\section{Hypergraphs}\label{sec:hypergraph}
A \emph{hypergraph} is a pair $( V,E)$, where $V$ is a set---the \emph{vertices}---and $E$ is a set of non-empty subsets of $V$---the \emph{hyperedges}.  For $k\geq 1$, a \emph{$k$-uniform hypergraph} is a hypergraph $(V,E)$ where all hyperedges have exactly $k$ elements.

\emph{Graphs} coincide with hypergraphs in which all hyperedges have size at most $2$, while \emph{simple graphs} are the $2$-uniform hypergraphs.  Many graph-theoretic concepts extend to hypergraphs in reasonably obvious ways.  
\begin{itemize}
\item An \emph{$n$-cycle} is a sequence $v_0,e_0,v_1,e_1,\dots,v_{n-1},e_{n-1}$ alternating between distinct vertices $v_0,\dots,v_{n-1}$ and distinct hyperedges $e_0,\dots,e_{n-1}$, such that $v_i\in e_i\cap e_{i+1}$ (with addition in the subscript taken modulo $n$).
\item An \emph{$\ell$-colouring} of a hypergraph $\langle V;E\rangle$ is a function $\gamma:V\to\{0,1,\dots,\ell-1\}$ such that $|\gamma(e)|\geq 2$ for each $e\in E$.
\item the \emph{chromatic number} $\chi$ of a hypergraph $\langle V;E\rangle$ is the smallest $\ell$ for which $\langle V;E\rangle$ is $\ell$-colourable.  
In other words, the chromatic number is the smallest number of colours required to colour the vertices in such a way that no hyperedge is monochromatic.
\item Two vertices $z$ and $w$ of $V$ are \emph{adjacent} if they belong to a common hyperedge, and are \emph{connected} if there is a sequence $z= v_0, v_1, v_2, \dots , v_k = w$ of vertices of $V$ in which $v_{i-1}$ is adjacent to $v_i$, for $i=1, 2,\dots , k$.  A \emph{connected hypergraph} is a hypergraph where every pair of vertices is connected.
\item A \emph{hyperforest} is a hypergraph without cycles, and \emph{hypertree} is a connected hyperforest.
\end{itemize}

Let $k$ be at least as large as the maximal hyperedge cardinality of a hypergraph $\mathbb{H}=(V,E)$.  Then $\mathbb{H}$ may be considered as a relational structure $\langle V;r_E\rangle$ with a single $k$-ary relation~$r_E$ by treating each hyperedge $\{v_1,\dots,v_\ell\}$ (where $\ell\leq k$) as the family of $k$-tuples $\{(v_{i_1},\dots,v_{i_k})\mid \{i_1,\dots,i_k\}=\{1,\dots,\ell\}\}$.  We call such a structure a \emph{$k$-hypergraph structure}.  
\begin{eg}\label{eg:K2}
The simple graph $\mathbb{K}_2$ with edge $\{0,1\}$.  Treated as a $2$-hypergraph structure, $\mathbb{K}_2$ is 
\[
\langle \{0,1\};\{(0,1),(1,0)\}\rangle.
\]  
As a $3$-hypergraph structure $\mathbb{K}_2$ is 
\[
\langle \{0,1\};\{(0,0,1),(0,1,0),(1,0,0),(1,1,0),(1,0,1),(0,1,1)\}\rangle.
\]
\end{eg}
The class of all $k$-hypergraph structures is a universal Horn class, definable by the following \emph{set-equivalence} universal Horn sentences.
\begin{description}
\item[Set equivalence] $(x_{i_1},\dots,x_{i_k})\in r\rightarrow (x_{j_1},\dots,x_{j_k})\in r$, if $\{x_{i_1},\dots,x_{i_k}\}=\{x_{j_1},\dots,x_{j_k}\}$.
\end{description}
The class of $k$-uniform hypergraphs (as model-theoretic structures) is a subclass of the $k$-hypergraph structures, defined by adjoining the following \emph{uniformity} laws.
\begin{description}
\item[Uniformity] $(x_1, \dots, x_{k})\in r \rightarrow  x_i\ne x_j$ whenever $i\neq j$ and $i,j\in \{1,\dots,k\}$.
\end{description}
The class of loop-free hypergraphs (that is, with no singleton hyperedges) is a subclass obtained by adjoining the single universal Horn sentence 
$(x,\dots,x)\notin r$; clearly $k$-uniform hypergraphs are loop-free, except in the degenerate case of $k=1$.

\begin{remark}
By default we choose the arity $k$ to equal the maximal cardinality of any hyperedge in $\mathbb{H}$.  Our methods cover the case where $k$ is strictly larger than this, however the statement of results will be different.  Example \ref{eg:K2} illustrates the difference, as the homomorphism problem in the case of $k=2$ is the tractable problem of graph $2$-colouring  but is \texttt{NP}-complete problem \textsf{+NAE3SAT} when $k=3$.
\end{remark}

The notion of \emph{induced subhypergraph} in the next definition coincides with the model-theoretic notion of substructure.
\begin{defn}
A hypergraph $\mathbb{G}'=\langle V'; E'\rangle$ is an \emph{induced subhypergraph} of $\mathbb{G}=\langle V;E\rangle$ if $V'\subseteq V$ and $E'=\{e \cap {V'}\ |\ e\in E\}$.  
\end{defn}

The homomorphism notion also agrees with the model-theoretic homomorphism when both $\mathbb{G}$ and $\mathbb{G}'$ are considered as a $k$-hypergraph structures.
\begin{defn}
For any pair of hypergraphs $\mathbb{G}=(V,E)$ and $\mathbb{G}'=(V', E')$, a map $f:V\to V'$ is a \emph{homomorphism} if for each $e\in E$, the set $f(e)=\{f(v)\mid v\in e\}$ is an element of $E'$.
\end{defn}
As usual, $\mathbb{G}\rightarrow \mathbb{G}'$ will denote the statement ``there exists a homomorphism from $\mathbb{G}$ to $\mathbb{G}'$'' and $\mathbb{G}\not\rightarrow \mathbb{G}'$ will denote its negation.  
\begin{eg}\label{eg:complete}
Let $\mathbb{K}_{n}^{(k)}$ denote the loop-free hypergraph on $n$ points $\{0,1\dots,n-1\}$ and whose hyperedge set is the set of all subsets of $\{0,1,\dots,n-1\}$ of size between $2$ and $k$. 
Then a $k$-hypergraph structure $\mathbb{H}$ is $n$-colourable if and only if $\mathbb{H}\to \mathbb{K}_{n}^{(k)}$ \up(as a $k$-hypergraph structure\up).
\end{eg}
Note that the hypergraph $\mathbb{K}_{n}^{(2)}$ is the usual complete graph $\mathbb{K}_n$.

For a $k$-ary relation $r$, let us denote the \emph{set closure of $r$}, denoted $\set(r)$, to be the closure of $r$ under applications of the set equivalence laws: in other words, if $(s_1,\dots,s_k)\in r$ is a tuple, then we add the add tuple $(s_1',\dots,s_k')$ to $\set(r)$ whenever $\{s_1',\dots,s_k'\}=\{s_1,\dots,s_k\}$.  It is trivial that if $\mathbb{S}=\langle S;r^\mathbb{S}\rangle$ is a relational structure in the signature of a single $k$-ary relation $r$, then provided $r^\mathbb{S}$ has no constant tuples $(s,\dots,s)\in r^\mathbb{S}$, the structure $\mathbb{S}=\langle S;\set(r^\mathbb{S})\rangle$ is a $k$-hypergraph structure, and is a $k$-uniform hypergraph structure if $r^\mathbb{S}$ already satisfied the uniformity laws.  We write $\set(\mathbb{S})$ to denote the result of replacing $r^\mathbb{S}$ by $\set(r^\mathbb{S})$.  The following lemma is also trivial.
\begin{lem}\label{lem:setclosure}
Let $\mathbb{H}$ be a $k$-hypergraph structure and $\mathbb{S}=\langle S;r^\mathbb{S}\rangle$ be a relational structure with a single $k$-ary relation $r$.  Then $\mathbb{S}\rightarrow \mathbb{H}$ if and only if  $\set(\mathbb{S})\rightarrow\mathbb{H}$.
\end{lem}

The next theorem is essentially Theorem 5 of Feder and Vardi \cite{fedvar}, except that their result is proved relative to the class of all relational structures rather than $k$-hypergraph structures.  In the more general setting of \cite{fedvar}, what we have written as $\mathbb{H}_2^\sharp$ would not be a $k$-uniform hypergraph structure (but rather just some general relational structure of the same signature as $\mathbb{H}_2$) and the notion of cycle is more restrictive than the one we use.  The statement we give in Theorem \ref{thm:fedvar} follows immediately from \cite[Theorem~5]{fedvar} after an application of the $\set(\ )$ operator and Lemma \ref{lem:setclosure}.  
\begin{thm}\label{thm:fedvar}
Fix any positive integer $\ell$.  Let $\mathbb{H}_1$ and $\mathbb{H}_2$ be $k$-hypergraph structures such that there is no homomorphism from $\mathbb{H}_2$ to $\mathbb{H}_1$.  Then there is a $k$-uniform hypergraph structure $\mathbb{H}_2^\sharp$ such that 
\begin{enumerate}
\item $\mathbb{H}_2^\sharp\rightarrow \mathbb{H}_2$\up;
\item $\mathbb{H}_2^\sharp\not\rightarrow \mathbb{H}_1$\up;
\item  any cycle in $\mathbb{H}_2^\sharp$ has size greater than $\ell$.
\end{enumerate}
\end{thm}
The following theorem is due to Erd\H{o}s and Hajnal \cite{erdhaj}, but it follows immediately from Theorem \ref{thm:fedvar} and Example \ref{eg:complete}, by choosing $\mathbb{H}_1:=\mathbb{K}_n^{(k)}$ and $\mathbb{H}_2:=\mathbb{K}_{n+1}^{(k)}$.
\begin{thm}\label{thm:erdhaj}
For any $k\geq 2$ and $\ell,n>1$ there is a finite $k$-uniform hypergraph $\mathbb{H}$ such that $\mathbb{H}$ has no cycles of length less than $\ell$ and is not $n$-colourable. 
\end{thm}

A hyperedge for which at most one vertex is contained in more than one hyperedge is called a \emph{leaf}.  A routine variation of the standard argument for $2$-uniform hyperforests (that is, forests) shows that every finite hyperforest contains at least one leaf.

In the following we use the well known fact that a structure $\mathbb{S}$ lies in the universal Horn class of some finite structure $\mathbb{M}$ if and only if the following \emph{separation conditions} hold:
\begin{enumerate}
 \item[(SEP1)] there exists a homomorphism $\phi$ from $\mathbb{S}$ to $\mathbb{M}$;
 \item[(SEP2)] for all $x, y \in S$ with $x\not=y$, there exists a homomorphism $\psi$ from $\mathbb{S}$ to $\mathbb{M}$ satisfying $\psi(x)\ne \psi(y)$;
 \item[(SEP3)] for every relation $r$ in the signature, with arity $n$, if $(s_1,\dots,s_n)\in S^n\backslash r^\mathbb{S}$, then there exists a homomorphism $\gamma$ satisfying $(\gamma(s_1), \dots , \gamma(s_k))\notin M^n\backslash r^\mathbb{M}$.
\end{enumerate}
We  mention that if $\mathbb{S}$ is also finite then these conditions imply that $\mathbb{S}$ is isomorphic to an induced substructure of a finite direct power of $\mathbb{M}$, indeed it is easy to prove that $\mathbb{S}$ is isomorphic to a substructure of $\mathbb{M}^{\hom(\mathbb{S},\mathbb{M})}$, where $\hom(\mathbb{S},\mathbb{M})$ denotes the set of all homomorphisms from $\mathbb{S}$ to $\mathbb{M}$, which is finite if $\mathbb{S}$ and $\mathbb{M}$ are finite.

\begin{lem}\label{lem:hyperforest}
Let  $k\geq 3$ and $\mathbb{E}_k=\langle\{v_1, \dots, v_k\};\{\{v_1,\dots,v_k\}\}\rangle$ be the  hypergraph containing exactly one hyperedge.  If $\mathbb{E}_k$ is considered as a $k$-uniform hypergraph, then $\mathsf{SP}(\mathbb{E}_k)$ contains all $k$-uniform hyperforests, with all finite $k$-uniform hyperforests lying in  $\mathsf{SP}_{\rm fin}(\mathbb{E}_k)$.
\end{lem}
\begin{proof}
Every relational structure embeds into an ultraproduct of its finite substructures, it suffices to prove the lemma in the case of finite hyperforests.  Thus if we show that every finite $k$-uniform hyperforest lies in $\mathsf{SP}(\mathbb{E}_k)$, then it follows that every $k$-uniform hyperforest lies in $\mathsf{SP}(\mathbb{E}_k)$.  Let $\mathbb{F}=\langle F;r^\mathbb{F}\rangle$ be a finite $k$-uniform hyperforest, with $r^\mathbb{F}$ the fundamental $k$-ary relation.  

We proceed by induction on the number,~$n$, of hyperedges of $\mathbb{F}$.  For simplicity, we will assume that there are no isolated points as it is close to trivial to extend the separation conditions below to include these.

The base case with $n=1$ is trivial, so assume that every $k$-uniform hyperforest with at most $n-1$ hyperedges belongs to the class~$\mathsf{SP}(\mathbb{E}_k)$.  Let $e=\{u_1, \dots, u_k\}$ be a leaf in $\mathbb{F}$.  At most one vertex in $e$ lies in any other hyperedge; if it it exists denote it by $u$, which otherwise is a symbol not equal to the label of any vertex.  Now let~$\mathbb{F}_{n-1}$ be the subhyperforest induced by removing the elements $\{u_1, \dots, u_k\}\setminus \{u\}$ from $\mathbb{F}$.  
By the induction hypothesis, we have $\mathbb{F}_{n-1}\in \mathsf{SP}(\mathbb{E}_k)$ and so conditions (SEP1)--(SEP3) hold.
We now show that $\mathbb{F}\in \mathsf{SP}(\mathbb{E}_k)$.  First we show that every homomorphism $\phi\colon\mathbb{F}_{n-1}\to\mathbb{E}_k$ extends to a homomorphism $\phi^+\colon\mathbb{F}\to\mathbb{E}_k$, giving~(SEP1).  Simply define $\phi^+(v):=\phi(v)$ for all $v\in \mathbb{F}_{n-1}$ (in particular $u$ is sent to $\phi(u)$) and send each element in $\{u_1, \dots, u_k\}\setminus \{u\}$ to a different element of $\{v_1, \dots, v_k\}\setminus\phi(u)$.  When $u\notin F$, this simply means we map $\{u_1, \dots, u_k\}$ onto $\{v_1, \dots, v_k\}$, giving $k!$ possible choices for $\phi^+$.  When $u\in F$ there are $(k-1)!$ choices for $\phi^+$.

Now let $e'=(w_1, \dots, w_k)\notin r^{\mathbb{F}}$ be any non-hyperedge of $\mathbb{F}$ (for verifying (SEP3)) and let $x, y\in F$ with $x\neq y$ (for verifying (SEP2)). There are two cases to consider. 

\emph{Case} $1$: If $\{w_1, \dots, w_k\}$ is a subset of $F_{n-1}$, then (SEP3) in the case of $\mathbb{F}_{n-1}$ guarantees the existence of a homomorphism $\gamma\colon\mathbb{F}_{n-1}\to\mathbb{E}_k$ mapping $e'$ strictly into $\{v_1,\dots,v_k\}$ (that is, to a non-hyperedge of $\mathbb{E}_k$).  Then $\gamma^+\colon\mathbb{F}\to\mathbb{E}_k$ is the desired homomorphism for (SEP3).  The same technique applies if the pair $x,y$ with $x\neq y$ both lie in $F_{n-1}$, giving (SEP2). 

\emph{Case} $2$: If $e'=(w_1, \dots, w_k)$ contains an element $w_j$ not in $F_{n-1}$, then $w_j$ is an element of $\{u_1, \dots, u_k\}\backslash\{u\}$.  If $|\{w_1,\dots,w_k\}|<k$ then any homomorphism from $\mathbb{F}$ to $\mathbb{E}_k$ will fail to map $\{w_1,\dots,w_k\}$ onto $\{v_1,\dots,v_k\}$: since there exists a homomorphism $\phi$ from $\mathbb{F}_{n-1}$ by (SEP1), the homomorphism $\phi^+$ completes the argument for (SEP3) in this subcase.  

Now assume that $|\{w_1,\dots,w_k\}|=k$, and observe that since $\{w_1, \dots, w_k\}$ is not a hyperedge of $\mathbb{F}$, it cannot be equal to the hyperedge $\{u_1, \dots, u_k\}$, and so contains at least one element from $F_{n-1}$ other than $u$. Without loss of generality we may assume that $w_1$ is such an element.  Note that  $w_1\neq w_j$ because $w_j\notin F_{n-1}$ by assumption.  Fix any homomorphism $\phi\colon\mathbb{F}_{n-1}\to\mathbb{E}_k$ separating $u$ from $w_1$ (which exists because~$\mathbb{F}_{n-1}$ satisfies~(SEP2)) and define a homomorphism $\phi'$ from $\mathbb{F}$ to $\mathbb{E}_k$ in the following way.  For $a\in {F}_{n-1}$, define $\phi'(a):=\phi(a)$ and define $\phi'(w_j):=\phi(w_1)$.  Finally, let $\phi'$ send each element in $\{u_1, \dots, u_k\}\setminus \{u, w_j\}$ to a different element of $\{v_1, \dots, v_k\}\setminus \{\phi'(u), \phi'(w_j)\}$. Clearly, the map $\phi'$ is a homomorphism, and it maps~$e'$ to a non-hyperedge since $\phi'(w_1)=\phi'(w_j)$ implies that $\{\phi'(w_1), \dots , \phi'(w_k)\}\subsetneq \{v_1, \dots, v_k\}$.  Thus (SEP3) holds.

To separate the pair $x\neq y$ when at least one of $x,y$ is not in $F_{n-1}$, simply take any homomorphism $\phi\colon\mathbb{F}_{n-1}\to\mathbb{E}_k$ and note that a large number of the $(k-1)!$ choices for $\phi^+$ will separate $x$ from $y$, establishing (SEP2).
\end{proof}

\begin{lem}\label{lem:NAE}
Let $k>\ell>1$ and $\mathbb{E}=\langle \{v_1,\dots,v_\ell\};\{\{v_1,\dots,v_\ell\}\}\rangle$ be a hypergraph with exactly one hyperedge.  Then if $\mathbb{E}$ is treated as a $k$-hypergraph structure, the $\mathsf{SP}$-class of $\mathbb{E}$ includes the $k$-uniform hypergraph $\mathbb{E}_k:=\langle \{u_1,\dots,u_k\};\{\{u_1,\dots,u_k\}\}\rangle$.
\end{lem}
\begin{proof}
The homomorphisms from $\mathbb{E}_k$ to $\mathbb{E}$ coincide with the surjective maps from $\{u_1,\dots,u_k\}$ onto $\{v_1,\dots,v_\ell\}$.  It is trivial that such maps exist (SEP1) and that any pair of points may be separated by a suitable map, given that $\ell\geq 2$ (SEP2).  For (SEP3), a non-hyperedge of $\mathbb{E}_k$ as a $k$-hypergraph structure is any $k$-tuple that has a repeat, say,  $(u_{i_1},\dots,u_{i_k})$ with $|\{u_{i_1},\dots,u_{i_k}\}|=j$ for some $j<k$.  If $j<\ell$ then every homomorphism maps $(u_{i_1},\dots,u_{i_k})$ to a non-hyperedge.  If $j\geq \ell$, then map $\{u_{i_1},\dots,u_{i_k}\}$ onto $\{v_1,\dots,v_{\ell-1}\}$ and all remaining $k-j$ elements of $\{u_1,\dots,u_k\}$ onto $\{v_\ell\}$.
\end{proof}
Consider the following $\mathsf{SP}$-classes generated by a single hypergraph.
\begin{itemize}
\item Let $\mathbb{G}_1=\langle\{1\}; \varnothing\rangle$ be the edgeless hypergraph on one vertex and let $\mathscr{Q}_{1(k)}=\mathsf{SP}(\mathbb{G}_1)$, where $\mathbb{G}_1$ is treated as a $k$-hypergraph structure.

\item Let $\mathbb{G}_2=\langle\{1, 2\}; \varnothing\rangle$ be the edgeless hypergraph on two vertices and let $\mathscr{Q}_{2(k)}=\mathsf{SP}(\mathbb{G}_2)$, where $\mathbb{G}_2$ is treated as a $k$-hypergraph structure.
\end{itemize}

\begin{proof}[Proof of Theorem \ref{thm:main}]
The following argument applies whenever~$\mathscr{K}$ is a class of loopfree hypergraphs of finite bounded hyperedge cardinality $c$ and $k\geq c$; in the theorem statement, the class $\mathscr{H}$ is $\mathsf{SP}(\mathscr{K})$.  If all members of $\mathscr{K}$ have no hyperedges, then the universal Horn class generated by $\mathscr{K}$ is equal to either $\mathscr{Q}_{1(k)}$ or~$\mathscr{Q}_{2(k)}$.  Now assume that $\mathscr{K}$ contains a hypergraph with at least one hyperedge.  The case where $k=c=2$ is covered by Caicedo \cite{cai}, so now assume that $k>2$.  Let~$\mathbb{H}$ be a hypergraph in $\mathscr{K}$ containing a hyperedge $e$.  Assume that $e$ has minimal cardinality amongst the hyperedges of $\mathbb{H}$, and let $\mathbb{E}$ denote the induced substructure on the elements of $e$, which consists of a single hyperedge $e$ and lies in $\mathsf{S}(\mathscr{K})$.  As $k\geq 3$ we find by Lemma \ref{lem:NAE} that the singleton hyperedge $k$-uniform tree $\mathbb{E}_k$ lies in $\mathsf{SP}(\mathbb{E})\subseteq \mathsf{SP}(\mathscr{K})$.  We will show that $\mathsf{SP}(\mathscr{K})$ is not definable by any universal sentence by showing that for every $n\in {\mathbb N}$ there exists a $k$-uniform hypergraph $\mathbb{U}_n$ such that the following properties hold.
\begin{itemize}
\item  The structure $\mathbb{U}_n$ is not in $\mathsf{SP}(\mathscr{K})$.
\item Every $n$-generated substructure of $\mathbb{U}_n$ belongs to $\mathsf{SP}(\mathbb{E})\subseteq \mathsf{SP}(\mathscr{K})$. 
\end{itemize}
For $n\in\mathbb{N}$, Theorem~\ref{thm:erdhaj} shows that there exists a finite $k$-uniform hypergraph $\mathbb{U}_n$ with chromatic number strictly greater than that of $\mathscr{K}$ and has no cycles of length less than $n+1$. 
This necessarily places $\mathbb{U}_n$ outside of $\mathsf{SP}(\mathscr{K})$, as there are no homomorphisms from $\mathbb{U}_n$ into any member of $\mathscr{K}$. 
However, an $n$-element induced substructure of $\mathbb{U}_n$ is a $k$-uniform hyperforest, so lies in $\mathsf{SP}(\mathbb{E}_k)\subseteq \mathsf{SP}(\mathbb{E})\subseteq \mathsf{SP}(\mathscr{K})$, by Lemma~\ref{lem:hyperforest}.  
\end{proof}

\begin{remark}
We may also extend a result of Trotta  from the class of simple graphs to the class of hypergraphs.  Trotta \cite[Theorem 2.4]{tro} showed that a simple graph is \emph{standard} in the sense of 
Clark et al.~\cite{CDHPT} if and only if it either has no edges or consists only of disjoint unions of isolated points and single edge graphs.  A version of Theorem \ref{thm:fedvar} is used (via a construction from \cite{CDJP}) to show nonstandardness for any graph not equal to a disjoint union of complete bipartite graphs; see the proof of Theorem 3.9 in \cite{tro}.  An identical argument for hypergraphs, shows that for $k\geq 3$, a $k$-hypergraph structure is standard if and only if it has no hyperedges.  This also positively answers Problem~3 of~\cite{CDJP} in the particular case of hypergraphs.
\end{remark}

\section{Universal Horn class lattices are continuum in cardinality}\label{lem:continuum}
It is shown in Caicedo \cite{cai} that there are continuum many universal Horn classes of graphs.  The argument has an easy adaptation to the present setting, but we instead follow a substantial extension of  Caicedo's result proved by Bonato~\cite{bon}: any interval in the homomorphism order on simple graphs (above the bipartite graphs) contains continuum many universal Horn classes.  Indeed, Bonato's very short argument shows that it suffices to show that intervals in the homomorphism order satisfy a density property.  We mention that a density result corresponding to that cited by Bonato is known for hypergraphs---Ne\v{s}et\v{r}il \cite[Theorem~1.4]{nes}---however the proof there makes intrinsic use of hypergraphs of increasingly large hyperedge cardinality and so is not available here.  In the proof of the following theorem we find an alternative proof of \cite[Theorem~1.4]{nes} involving bounded hyperedge cardinality.  As usual, we write $\mathbb{A}\rightarrow \mathbb{B}$ to denote the existence of a homomorphism from~$\mathbb{A}$ to $\mathbb{B}$.
\begin{thm}\label{thm:continuum}
Let $\mathbb{G}_1$ and $\mathbb{G}_2$ be finite $k$-hypergraph structures, both containing at least one hyperedge.  For $i=1,2$, let $\mathscr{U}_i$ denote the universal Horn class of all $k$-hypergraph structures admitting a homomorphism into $\mathbb{G}_i$.  If  $\mathbb{G}_1\rightarrow\mathbb{G}_2$ but $\mathbb{G}_2\not\rightarrow\mathbb{G}_1$ \up(equivalently, $\mathbb{G}_1\in \mathscr{U}_2$ but $\mathbb{G}_2\notin \mathscr{U}_1$\up), then there is a continuum of universal Horn classes between $\mathscr{U}_1$ and $\mathscr{U}_2$.  
\end{thm}
\begin{proof}
The argument of Bonato (in the proof of \cite[Proposition 6]{bon}) applies immediately, provided we can show that the homomorphism order is dense between $\mathscr{U}_1$ and $\mathscr{U}_2$.  
It suffices to show that there is a hypergraph $\mathbb{H}$ lying strictly between $\mathbb{G}_1$ and $\mathbb{G}_2$ in the homomorphism order.

Because both $\mathbb{G}_1$ and $\mathbb{G}_2$ contain a hyperedge, it follows by Lemmas \ref{lem:hyperforest} and \ref{lem:NAE} (and (SEP1)) that all hyperforests admit a homomorphism into both $\mathbb{G}_1$ and $\mathbb{G}_2$.  Then the property $\mathbb{G}_2\not\rightarrow \mathbb{G}_1$ shows that $\mathbb{G}_2$ does not admit a homomorphism into any hyperforest.  Let $\mathbb{G}_2^\sharp$ be the $k$-uniform hyperforest shown to exist in Theorem \ref{thm:fedvar}, with $\ell:=|G_2|+1$, and let $\mathbb{H}$ be the $k$-hypergraph structure $\mathbb{G}_2^\sharp\cup \mathbb{G}_1$.  Then $\mathbb{H}\rightarrow\mathbb{G}_2$.  Also, $\mathbb{G}_1\rightarrow\mathbb{H}$ but $\mathbb{H}\not\rightarrow\mathbb{G}_1$.  Thus, it remains to show that $\mathbb{G}_2\not\rightarrow\mathbb{H}$.  Now, at least one component of $\mathbb{G}_2$ does not homomorphically map into $\mathbb{G}_1$, by assumption.  To complete the proof, assume for contradiction that this component homomorphically maps into the $\mathbb{G}_2^\sharp$ component of $\mathbb{H}$.  By property (3) of Theorem~\ref{thm:fedvar}, this component maps into a sub-hypertree of $\mathbb{G}_2^\sharp$, contradicting the fact that $\mathbb{G}_2$ does not have a homomorphism into any hyperforest.
\end{proof}

\section{Axiomatisability at the finite level}\label{sec:finite}
Let $\mathscr{H}$ be an $\textsf{SP}_{\rm fin}$-closed class of \emph{finite} hypergraphs of bounded chromatic number (and hyperedge cardinality).  The proof of Theorem \ref{thm:main} shows that unless $\mathscr{H}$ consists only of disjoint unions of complete bipartite graphs, no finite set of universal Horn sentences can axiomatise $\mathscr{H}$ amongst finite structures.   A classical model-theoretic intuition (namely, the \emph{\textsf{SP}-Preservation Theorem}; see McNulty~\cite{mcn}) would then imply that no first order sentence can define $\mathscr{H}$.  In the restriction to finite structures however, there is no completely general \textsf{SP}-Preservation Theorem---see \cite[Example 4.3]{CDJP}---though the possibility of such a result remains an open problem in the case of relational signatures; see \cite[Problem~1]{alegur} and \cite[\S2.4.2]{ros}.
In this section we provide an argument that shows that the intuition is nevertheless correct in the case of hypergraphs: $\mathscr{H}$ cannot be defined by any first order sentence at the finite level.  

We prove a  more general result, deducing the finite level version of Theorem~\ref{thm:main} as a corollary.  For any relational structure $\mathbb{A}$, we let $\overline{\mathbb{A}}$ be the graph on the same underlying set $A$, with edge relation obtained by placing an undirected edge between $a,b\in A$ whenever $a$ and $b$ appear together in the tuple of one of the relations of~$\mathbb{A}$.  When~$\mathbb{A}$ is a graph we have $\overline{\mathbb{A}}=\mathbb{A}$.  Define the distance $d_{\mathbb{A}}(a,b)$ between two vertices $a,b$ in $\mathbb{A}$ to be the length of the shortest path of edges between $a$ and $b$ in~$\overline{\mathbb{A}}$.    Note that the distance may be infinite, which we denote by $d(a,b)=\infty$.  When $a=b$ the distance $d(a,b)$ is $0$.  Let the \emph{$n$-ball}~$B_n(a)$ of $a$ in $\mathbb{A}$ be the set  $\{x\in A\mid d(x,a)\leq n\}$ and let~$\mathbb{B}_n(a)$ denote the induced substructure of $\mathbb{A}$ on $B_n(a)$.  Note that the distance of any $b\in B_n(a)$ from $a$ in~$\mathbb{B}_n(a)$ remains equal to the distance from $b$ to $a$ in $\mathbb{A}$, but in general the distance between two elements of $B_n(a)$ distinct from 
$a$ may be larger in $\mathbb{B}_n(a)$ than in $\mathbb{A}$.  The following easy observation generalises this.
\begin{obs}\label{obs:distance}
Let $b,c$ be elements of an $n$-ball $B_n(a)$ in $\mathbb{A}$ lying at  distance $j$ and~$j'$ from $a$ respectively.  If the distance $\delta$ from $b$ to $c$ in $\mathbb{A}$ is at most $2n-j-j'$, then the distance $d_{\mathbb{B}_n(a)}(b,c)$ from $b$ to $c$ in $\mathbb{B}_n(a)$ is also $\delta$.
\end{obs}
\begin{proof}
Consider a path from $b$ to $c$ in $\overline{\mathbb{A}}$ of length $\delta$.  The first $n-j$ elements are distance at most $n-j+j=n$ from $a$, and the final $n-j'$ are distance at most $n-j'+j'$ from $a$.  Thus all lie in $B_n(a)$ showing that the distance from $b$ to $c$ is $\delta$ in $\mathbb{B}_n(a)$ as well.
\end{proof}
 The \emph{boundary} of an $n$-ball $\mathbb{B}_n(a)$ is the set of elements that are distance exactly~$n$ from $a$.  Note that the $n$-ball $\mathbb{B}_n(a)$ can have empty boundary, such as if $n\geq 1$ and $a$ is an isolated point.
\begin{thm}\label{thm:finitelevel}
Let $\mathscr{K}$ \!be an $\mathsf{S}$-closed class of finite structures of some relational signature such that for all $n$ there exists a finite structure $\mathbb{S}_n$ with the following properties\up:
\begin{itemize}
\item   $\mathbb{S}_n\notin \mathscr{K}$\!\up;
\item  The disjoint union of any finite number of copies of $n$-balls in $\mathbb{S}_n$ lies in~$\mathscr{K}$\!.
\end{itemize}
Then $\mathscr{K}$ \!cannot be defined amongst finite structures by any first order sentence.
\end{thm} 
\begin{proof}
We use a standard Ehrenfeucht-Fra\"{\i}ss\'e game argument: see Libkin~\cite{lib}.  For each $k$, let $n$ be any integer greater than $2^{k+1}$  and let $\mathbb{H}_k$ consist of the disjoint union of  $k$ copies of every $n$-ball $\mathbb{B}_n(a)$, for every $a\in S_n$.  Let $\mathbb{G}_k$ denote the disjoint union of $\mathbb{H}_k$ with $\mathbb{S}_n$.  The second condition on $\mathscr{K}$ trivially shows that $\mathbb{H}_k\in \mathscr{K}$.  Because $\mathscr{K}$ is $\mathsf{S}$-closed, the complement class to $\mathscr{K}$ is closed under extensions and contains $\mathbb{S}_n$, by the second condition on $\mathscr{K}$. Thus $\mathbb{G}_k\notin \mathscr{K}$.  We make frequent reference to the boundaries of $n$-ball components, and to the $\mathbb{S}_n$ component, which we define to have no boundary.

We show how Duplicator has a winning strategy against Spoiler in a $k$-round 
Ehrenfeucht-Fra\"{\i}ss\'e game on the pair $\mathbb{G}_k$, $\mathbb{H}_k$. 
After $i\in\{0, 1, \dots, k\}$ rounds of the game, the players have selected points $g_1,\dots,g_i$ from $\mathbb{G}_k$ and $h_1,\dots,h_i$ from~$\mathbb{H}_k$, and Duplicator has not lost if the induced substructures on these points are isomorphic.  
At each round $i\in\{0, 1, \dots, k\}$, we will say that the distance between two elements $x$ and $x'$ in $\mathbb{G}_k$ or $\mathbb{H}_k$ is \emph{$\lrge_i$} if $d(x,x')\geq 2^{k-i+1}$.  The basic idea is that whenever two points $x$, $x'$ are distance at least $\lrge_i$, then a selection of any third point will be at least $\lrge_{i+1}$ from one of $x$ and $x'$ (this follows because $\lrge_i/2 =2^{k-i+1}/2=2^{k-(i+1)+1}=\lrge_{i+1}$) and that at the end of the game (when $i=k$) the value of $\lrge_k$ is greater than $1$.  For similar arguments, see Libkin~\cite[Chapter~$3$]{lib}.

It is convenient to fix some isomorphisms between any two copies of an $n$-ball, and also between each copy of each $n$-ball $\mathbb{B}_n(a)$ component in either of $\mathbb{G}_k$ or $\mathbb{H}_k$ and the actual substructure $\mathbb{B}_n(a)$ of the $\mathbb{S}_n$ component. Our strategy makes reference to these isomorphisms.  When Duplicator makes a move in response to Spoiler, she will first decide which component to play in---as determined by distances between elements---and once this is chosen, select the appropriate corresponding element---as determined by the fixed isomorphism.  Throughout the proof, we refer to ``corresponding element'' rather than make explicit reference to the fixed isomorphisms.
We will show, inductively, that Duplicator can not only maintain partial isomorphism but also preserve the following conditions at each of the rounds $i\in\{0, \dots, k\}$. For any $0<\ell,j<i$:
\begin{enumerate}[\quad \rm(1)]
\item the element $h_\ell$ in $\mathbb{H}_k$ is a corresponding element of $g_\ell$ in $\mathbb{G}_k$\textup;
\item if $d_{\mathbb{G}_k}(g_\ell, g_j)<2^{k-i+1}$, then $d_{\mathbb{H}_k}(h_\ell, h_j)=d_{\mathbb{G}_k}(g_\ell, g_j)$\textup;
\item if $d_{\mathbb{G}_k}(g_\ell, g_j)\geq 2^{k-i+1}$, then $d_{\mathbb{H}_k}(h_\ell, h_j)\ge2^{k-i+1}$\textup;
\item for $\ell<2^{k-i+1}$, the point $g_j$ is of distance $\ell$ from a boundary if and only if $h_j$ is of distance $\ell$ from a boundary.
\end{enumerate}

The base case holds vacuously. For the induction step, suppose that Duplicator has maintained isomorphism and the four conditions to the completion of round $i$.  
We assume by default that Spoiler is making his $({i+1})^{\rm st}$ move in $\mathbb{G}_k$, but note at key points how a similar argument would cover the case where his move is made in~$\mathbb{H}_k$.  

If Spoiler's selection for $g_{i+1}$ is equal to some previously played element $g_\ell$, where $\ell\leq i$, then Duplicator's response should be $h_\ell$.  Now assume that Spoiler selects an element not previously played.

Case $1$: Any previously played element is distance greater than or equal to $\lrge_{i+1}$ from $g_{i+1}$. \\
Case $1$(a): Spoiler chose $g_{i+1}$ from the $\mathbb{S}_n$ component of $\mathbb{G}_k$.  In this case, Duplicator selects a copy of the ball $\mathbb{B}_n(g_{i+1})$ that has no previously played points in it: after round $i$ there are at least $k-i$ unplayed copies remaining in $\mathbb{H}_k$.  To maintain the hypotheses, Duplicator can select $h_{i+1}$ to be the element corresponding to~$g_{i+1}$.  Case 1(a) does not occur if Spoiler is selecting in $\mathbb{H}_k$.

Case $1$(b): Spoiler chose $g_{i+1}$ in one of the $n$-ball components, a copy of $\mathbb{B}_{n}(a)$, where $a$ is some element in $\mathbb{S}_{n}$.  
Again, Duplicator finds an unused copy of $\mathbb{B}_n(a)$ in $\mathbb{H}_k$ and selects $h_{i+1}$ as the element corresponding to $g_{i+1}$.  All comparative distances are $\lrge_{i+1}$ for both $h_{i+1}$ and  $g_{i+1}$, so the hypotheses are maintained.  A symmetric argument applies when Spoiler is selecting in $\mathbb{H}_k$.

Case $2$. There exists some previously played element $g_{\ell}$ ($\ell<i$) that is distance $d<\lrge_{i+1}$ from $g_{i+1}$.  Let $\mathbb{B}$ denote the $n$-ball component of $\mathbb{H}_k$ containing~$h_\ell$: it is isomorphic to some specific $n$-ball $\mathbb{B}_n(a)$ of $\mathbb{S}_n$, for some $a$.  

Case 2(a).  Spoiler chose $g_{i+1}$ from an $n$-ball component of $\mathbb{G}_k$.  Then $g_\ell$ lies in this same $n$-ball component, and by Condition (1) on $g_\ell$ and $h_\ell$, this component is isomorphic to $\mathbb{B}$ and the choice of $h_{i+1}$ to correspond to $g_{i+1}$ is guaranteed.  Moreover, all comparative distances are identically equal or at least $\lrge_{i+1}$ for $h_{i+1}$ as for $g_{i+1}$, so the hypotheses are maintained.  A technicality here is if $\mathbb{B}$ contains some element $h_j$ for which $g_j$ lies in the $\mathbb{S}_n$ component so that $d_{\mathbb{G}_k}(g_{i+1},g_j)=\infty\geq\lrge_{i+1}$.  But then $d_{\mathbb{G}_k}(g_{\ell},g_j)=\infty$ also, so that $d(h_\ell,h_j)\geq \lrge_i$, and then the property $d(h_{i+1},h_\ell)< \lrge_{i+1}$ implies $d(h_{i+1},h_j)\geq \lrge_{i+1}$ by the triangle inequality.
In the dual to Case 2(a), a symmetric argument applies when Spoiler has selected $h_{i+1}$ near some $h_\ell$ for which $g_\ell$ lies in an $n$-ball component of $\mathbb{G}_k$.

Case 2(b).  Spoiler chooses $g_{i+1}$ from the $\mathbb{S}_n$ component.  We will show that an element corresponding to $g_{i+1}$ exists in $\mathbb{B}$, and that if Duplicator selects it as $h_{i+1}$, then the hypotheses are maintained.  

We first show that $g_{i+1}$ is contained in ${B}_n(a)$, the ball within $\mathbb{S}_n$ isomorphic to the component containing $h_\ell$.  Now, there is no boundary in the~$\mathbb{S}_n$ component, so Condition (4) implies that $h_\ell$ is at least $\lrge_\ell$ from the boundary of $\mathbb{B}$.  Then the distance $\epsilon$ from $h_\ell$ to the centre of $\mathbb{B}$ (the element corresponding to $a$) is at most $n-\lrge_\ell$.  Now the distance from $g_\ell$ to the point $a$ in $\mathbb{S}_k$ is exactly $\epsilon$ also, as $h_\ell$ corresponds to $g_\ell$ under the fixed isomorphism from $\mathbb{B}$ to $\mathbb{B}_n(a)$.  Hence the distance from $g_{i+1}$ to $a$ is at most $n-\lrge_\ell+\lrge_{i+1}\leq n-\lrge_{i+1}<n$, so that $g_{i+1}$ lies within the $n$-ball $\mathbb{B}_n(a)$ and a corresponding element $h_{i+1}$ from $\mathbb{B}$ can be selected.  Moreover $h_{i+1}$ lies at least $\lrge_{i+1}$ from the boundary, so that both Conditions (1) and (4) hold.   
Observation~\ref{obs:distance} now shows that $d_{\mathbb{H}_k}(h_{i+1}, h_\ell)=d_{\mathbb{G}_k}(g_{i+1},g_\ell)$ as well.  
In the dual case where Spoiler is choosing $h_{i+1}$ near $h_\ell$ in $\mathbb{H}_k$ and $g_\ell$ lies in the $\mathbb{S}_n$ component, then the choice of $g_{i+1}$ by Duplicator is immediate: use the fixed isomorphism from the component $\mathbb{B}$ to $\mathbb{B}_n(a)$, with all issues relating to distances now identical to the case just detailed.

It now remains to verify that Conditions $(2)$ and~$(3)$ are maintained for $g_{i+1}$ in comparison to any other element $g_j$ with $j\neq \ell$ and $j\leq i$.

Case $2$(b)(i): If $g_j$ is distance strictly less than $\lrge_{i}$ from $g_{\ell}$, then Condition ($2$) of the hypothesis tells us that $d_{\mathbb{H}_k}(h_\ell, h_j)=d_{\mathbb{G}_k}(g_\ell, g_j)$.  By the triangle inequality, the distance from $g_{i+1}$ to $g_j$ in $\mathbb{G}_k$ is at most $d_{\mathbb{G}_k}(g_{i+1}, g_\ell)+d_{\mathbb{G}_k}(g_\ell, g_j)\leq \lrge_{i+1}+\lrge_i= 2n-(n-\lrge_{i+1})-(n-\lrge_i)$ because $n>\lrge_0$ and $j\neq \ell$ implies $i\geq 1$.  Then Observation \ref{obs:distance} shows that $d_{\mathbb{H}_k}(h_{i+1}, h_j)=d_{\mathbb{G}_k}(g_{i+1}, g_j)$ and both Condition (2) and (3) are maintained in this case for $g_j$.
%

Case $2$(b)(ii): If $g_j$ is distance greater than or equal to $\lrge_{i}$ from $g_\ell$, then condition ($1$) of the hypothesis tells us that the distance of $h_j$ from $h_\ell$ in $\mathbb{H}_k$ is also at least $\lrge_i$.  Recall that  $d_{\mathbb{H}_k}(h_\ell, h_{i+1})=d_{\mathbb{G}_k}(g_\ell, g_{i+1})\leq\lrge_{i+1}$.  Then the triangle inequality and the property $\lrge_{i+1}=\lrge_i/2$ imply that both $d_{\mathbb{G}_k}(g_{i+1},g_j)$ and $d_{\mathbb{H}_k}(h_{i+1},h_j)$ are at least $\lrge_{i+1}$, showing that Conditions (2) and (3) are again maintained.

Finally we note that these conditions imply that the map $g_j\mapsto h_j$ is an isomorphism from the induced substructure on $\{g_1,\dots,g_{i+1}\}$ to $\{h_1,\dots,h_{i+1}\}$.  Conditions (2) and (3) show that this function is a bijection.  Assume that $(g_{i_1},\dots,g_{i_k})\in r$ is some hyperedge in the induced substructure on $\{g_1,\dots,g_{i+1}\}$.  Then all distances between elements of $g_{i_1},\dots,g_{i_k}$ are at most $1$.  Hence, by Condition (2), the same is true for $h_{i_1},\dots,h_{i_k}$.  Hence all of $h_{i_1},\dots,h_{i_k}$ lie in the same $n$-ball component $\mathbb{B}$.  Also, each $h_{i_j}$ is a corresponding element to $g_{i_j}$,  under the one fixed isomorphism from $\mathbb{B}$.  Because this is an isomorphism, the tuple $(h_{i_1},\dots,h_{i_k})$ lies in $r$ within $\mathbb{H}_k$, as required.
\end{proof}
\begin{cor}\label{cor:nfa}
Let $\mathscr{H}$ be an $\mathsf{SP}_{\rm fin}$-closed class of loop-free $k$-hypergraph structures of bounded chromatic number.  If $k=2$ and $\mathscr{H}$ contains a graph that is not a disjoint union of complete bipartite graphs, or if $k>2$ and at least one member of $\mathscr{H}$ has a hyperedge, then $\mathscr{H}$ has no finite axiomatisation in first order logic amongst finite structures.
\end{cor}
\begin{proof} 
Assume that $k=2$ and $\mathscr{H}$ contains a graph that is not a disjoint union of complete bipartite graphs, or $k>2$ and at least one member of $\mathscr{H}$ has a hyperedge.
Note that $\mathscr{H}$ coincides with the finite members of the universal Horn class $\mathsf{SPP}_{\rm u}(\mathscr{H})$.
Theorem \ref{thm:erdhaj} shows that there is a $k$-uniform hypergraph $\mathbb{V}_n$ not in $\mathsf{SPP}_{\rm u}(\mathscr{H})$ but whose cycles have length greater than $2n$.  Then an $n$-ball in $\mathbb{V}_{n}$ is a hyperforest.  A disjoint union of hyperforests is still a hyperforest, and hyperforests lie in $\mathsf{SPP}_{\rm u}(\mathscr{H})$: in the case of $k=2$, this is shown by Caicedo~\cite[Lemma~2]{cai}, while the $k
\geq 2$ case follows from Lemmas \ref{lem:hyperforest} and \ref{lem:NAE} above.  Then Theorem~\ref{thm:finitelevel} implies that $\mathscr{H}$ is not definable amongst finite structures by any first order sentence. 
\end{proof}

\section{Hardness}\label{sec:hard}
A well known result of Hell and Ne\v{s}et\v{r}il \cite{helnes} states that for a finite simple graph~$\mathbb{G}$, if~$\mathbb{G}$ is bipartite then $\mathbb{G}$-colourability of finite graphs can be decided in polynomial time, but otherwise is \texttt{NP}-complete.  The same dichotomy was recently established by the authors for universal Horn classes generated by finite simple graphs (with the same boundary of tractability).  In this section we show how to use this to provide an alternative path to Corollary \ref{cor:nfa} in the case of a universal Horn class generated by a finite loop-free hypergraph.  The basic idea is that if a class of structures can be defined in first order logic amongst finite structures, then it cannot be \texttt{NP}-complete with respect to first order reductions---this follows from the known strict containment in $\texttt{AC}^0\subsetneq \texttt{L}\subseteq \texttt{NP}$; see Immerman~\cite{imm}.

\subsection{Background concepts}  We begin with some basic concepts relating to the algebraic method in constraint satisfaction problem complexity.  We give only the bare necessities for the arguments we need; see \cite{bar,JKN,lar} for further background information on these concepts and their relationship to constraint satisfaction problems.

For any relational structure $\mathbb{A}$, we let $\CSP(\mathbb{A})$ denote the computational problem of deciding if an input finite structure admits a homomorphism into $\mathbb{A}$ (the \emph{constraint satisfaction problem} over $\mathbb{A}$, or the \emph{$\mathbb{A}$-colourability problem}), while $\QMEM(\mathbb{A})$ is the computational problem of deciding if an input finite structure lies in the quasivariety of~$\mathbb{A}$: this is almost identical to the problem of deciding membership in the universal Horn class of $\mathbb{A}$, as the universal Horn class and quasivariety differ by at most the one-element total structure, which has no impact on computational complexity, nor on the possible definability of the classes in first order logic.

A \emph{polymorphism} is a homomorphism $f\colon\mathbb{A}^n\to\mathbb{A}$, where $\mathbb{A}^n$ is the $n^{\rm th}$ direct power of~$\mathbb{A}$.  The polymorphism $f$ is said to be \emph{cyclic} if it satisfies  the equation $f(x_0,x_1,\dots,x_{n-1})=f(x_1,\dots,x_{n_1},x_0)$ for all $x_0,\dots,x_{n-1}\in A$, and idempotent if it satisfies $f(x,\dots,x)=x$ for all $x\in A$.  It is known that if $r$ is a relation definable on $\mathbb{A}$ by a $\exists \wedge$ formula (a conjunction of atomic formul{\ae}, with some variables existentially quantified), and $\langle A;r\rangle$ has no cyclic polymorphism, then~$\mathbb{A}$ has no cyclic polymorphism.  
\subsection{Hardness and nonfinite axiomatisability}
A fundamental contribution of Bulatov, Jeavons and Krokhin~\cite{BJK} was to show that if a finite relational structure $\mathbb{A}$ has no proper retracts and fails a particular special condition on its polymorphisms, then $\CSP(\mathbb{A})$ is \texttt{NP}-complete.  Using the results of Barto and Kozik \cite[Theorem~4.1]{barkoz1} and then Chen and Larose \cite[Lemma~6.4]{chelar}, the special condition can be stated as: there exists a cyclic polymorphism.  For our purposes we will use the following equivalent condition, also from \cite[Theorem~4.1]{barkoz1}:
for all primes $p>|A|$ there is a cyclic polymorphism of $\mathbb{A}$ arity $p$. 
The authors' All or Nothing Theorem \cite[Theorem~5.2]{hamjac} shows that the result of \cite{BJK} can be transfered to the membership problem for the quasivariety (and universal Horn class) of $\mathbb{A}$: if $\mathbb{A}$ has no cyclic polymorphism, then  $\QMEM(\mathbb{A})$ is  \texttt{NP}-complete with respect to first order reductions.  The main result of this section is a corollary of this.

\begin{thm}\label{thm:hard}
Let $k\geq c$ and $\mathbb{H}=\langle H;r\rangle$ be a finite loop-free $k$-hypergraph structure with maximal hyperedge cardinality $c$.  
\begin{itemize}
\item \up(Hell and Ne\v{s}et\v{r}il \cite{helnes}, Ham and Jackson \cite{hamjac}.\up) If $k=c=2$ or has no hyperedges at all, then $\CSP(\mathbb{H})$ and $\QMEM(\mathbb{H})$ are tractable if and only if $\mathbb{H}$ is bipartite.
\item Otherwise \up(that is, $k>2$ and there is a hyperedge of some cardinality $c\neq 0$\up), then $\CSP(\mathbb{H})$ and $\QMEM(\mathbb{H})$ are \texttt{NP}-complete with respect to first order reductions and neither can be defined by a first order sentence at the finite level.
\end{itemize}
\end{thm}
\begin{proof}
The first statement is trivial when there are no hyperedges at all.  When $k=c=2$, then the $\CSP(\mathbb{H})$ case is directly from \cite{helnes} and the $\QMEM(\mathbb{H})$ case is directly from \cite{hamjac}.  Now, assume that $k>2$ and $\mathbb{H}$ has a hyperedge $e$.  

Consider a hyperedge  $e$ of minimal cardinality $d\leq c$.  First assume that $d>2$ and consider the binary relation $\sim$ defined from $r$ by the formula
\[
\exists x_3\dots \exists x_d\ (x_1,x_2,x_3,\dots,x_{d-1},\stackrel{k-d+1}{\overbrace{x_d,\dots,x_d}})\in r
\]
in free variables $x_1,x_2$.
The formula $(x_1,x_2,x_3\dots,x_{d-1},\stackrel{k-d+1}{\overbrace{x_d,\dots,x_d}})\in r$ interprets in all hyperedges of cardinality $d$ and no others, so that $\sim$ is the graph consisting of $d$-cliques on each hyperedge of cardinality $d$.  As $d>2$ this graph is not bipartite, hence has no cyclic polymorphism by Barto, Kozik and Niven~\cite{BKN}.  Thus $\mathbb{H}$ has no cyclic polymorphism, as required.

Now assume that $d=2$, and let the elements in the hyperedge $e$ be denoted $0,1$.  We show that for any prime $p>|H|$, there is no cyclic polymorphism of arity $p$.  Assume for contradiction that such a $p$-ary polymorphism exists.  Let $s_0,\dots,s_{p-1}$ be a sequence in $\{0,1\}^p$ with the property that cyclically there is no run of $k$ consecutive $0$s, nor $k$ consecutive $1$s.  Such sequences are very easily seen to exist, given that $k>2$.    Let $a\in H$ be the value of $f(s_0,\dots,s_{p-1})$.  Because $f$ is cyclic we have the following equalities:
\[\begin{matrix}
f(s_0,& s_1, &s_2, &\dots,&s_{p-1})&=&a\\
f(s_1,&s_2,&s_3,&\dots,&s_0)&=&a\\
\vdots&\vdots&\vdots&\dots\phantom{,}&\vdots&&\vdots&\\
f(s_{k-1},&s_{k},&s_{k+1},&\dots,&s_{k-2})&=&a
\end{matrix}\]
Because there is no run of $k$ consecutive values in $s_0,\dots,s_{p-1}$ (treated cyclically), the tuples $(s_i,\dots,s_{i+k-1})$ forming columns on the left of the equalities lie in the fundamental relation $r$ on $\mathbb{H}$.
Hence as $f$ is a polymorphism, the constant tuple $(a,\dots,a)$ is in $r$.  But this contradictions the assumption that $\mathbb{H}$ was loop-free.  So no cyclic polymorphism of arity $p$ exists, as required.
\end{proof}
\begin{remark}
The All or Nothing Theorem of \cite{hamjac} actually shows a stronger result than what is stated in Theorem \ref{thm:hard}.  Whenever Theorem \ref{thm:hard} states \texttt{NP}-completeness of $\QMEM(\mathbb{H})$ the following holds: any class of finite hypergraphs $\mathscr{K}$ has \texttt{NP}-hard membership problem provided its members admit homomorphisms into $\mathbb{H}$ and that~$\mathsf{SP}_{\rm fin}(\mathbb{H})\subseteq\mathscr{K}$.
\end{remark}

\noindent {\bf Acknowledgement.}
The results above were originally developed in the context of $k$-uniform hypergraphs.  The authors thank Micha\l\  Stronkowski for observing that non-uniform hypergraphs of bounded hyperedge cardinality could also be considered as relational structures.

\end{document}